\documentclass[Physsubmission, Phys]{SciPost}

\pdfoutput=1
\hypersetup{
    colorlinks,
    linkcolor={red!50!black},
    citecolor={blue!50!black},
    urlcolor={blue!80!black}
}

\usepackage[bitstream-charter]{mathdesign}
\DeclareSymbolFont{usualmathcal}{OMS}{cmsy}{m}{n}
\DeclareSymbolFontAlphabet{\mathcal}{usualmathcal}

\newcommand{\BE}{\begin{equation} \begin{array}{c}}
\newcommand{\EE}{\end{array}\end{equation}}
\newcommand{\BQA}{\begin{equation} \begin{aligned}}
\newcommand{\EQA}{\end{aligned}\end{equation}}
\newcommand{\BT}{\begin{theorem}}
\newcommand{\ET}{\end{theorem}}

\newcommand{\bc}{\begin{center}}
\newcommand{\ec}{\end{center}}

\newcommand{\LX}{\Lambda}
\newcommand{\LXB}{\overline{\LX}}
\newcommand{\lX}{\lambda}

\usepackage{mathtools}
\usepackage{amsmath}
\usepackage{amstext}
\usepackage{amsthm}
\usepackage{amsfonts}

\newcommand{\nc}{\newcommand}

\nc{\one}{\mbox{\bf 1}}
\nc{\invtensor}{\underset{\leftarrow}{\otimes}}
\nc{\const}{\operatorname{const}}

\nc{\ad}{\operatorname{ad}}

\nc{\tr}{\operatorname{tr}}

\nc{\Gr}{\operatorname{Gr}}
\nc{\rGr}{\operatorname{rGr}}
\nc{\atyp}{\operatorname{atyp}}
\nc{\tp}{\operatorname{top}}
\nc{\rank}{\operatorname{rank}}
\nc{\corank}{\operatorname{corank}}
\nc{\codim}{\operatorname{codim}}
\nc{\sdim}{\operatorname{sdim}}
\nc{\mult}{\operatorname{mult}}
\nc{\ext}{\operatorname{ext}}
\nc{\tail}{\operatorname{tail}}
\nc{\howl}{\operatorname{howl}}
\nc{\spn}{\operatorname{span}}
\nc{\Sym}{\operatorname{Sym}}
\nc{\core}{\operatorname{core}}
\nc{\id}{\operatorname{id}}
\nc{\Id}{\operatorname{Id}}
\nc{\Ree}{\operatorname{Re}}
\nc{\hi}{\operatorname{hi}}
\nc{\htt}{\operatorname{ht}}
\nc{\at}{\operatorname{at}}
\nc{\str}{\operatorname{str}}
\nc{\Iso}{\operatorname{Iso}}
\nc{\Ker}{\operatorname{Ker}}
\nc{\rker}{\operatorname{rKer}}
\nc{\im}{\operatorname{Im}}
\nc{\osp}{\mathfrak{osp}}
\nc{\sgn}{\operatorname{sgn}}
\nc{\F}{\operatorname{F}}
\nc{\Mod}{\operatorname{Mod}}
\nc{\DS}{\operatorname{DS}}
\nc{\Soc}{\operatorname{Soc}}
\nc{\Inj}{\operatorname{Inj}}
\nc{\Hom}{\operatorname{Hom}}
\nc{\End}{\operatorname{End}}
\nc{\supp}{\operatorname{supp}}
\nc{\Card}{\operatorname{Card}}
\nc{\Ann}{\operatorname{Ann}}
\nc{\Ind}{\operatorname{Ind}}
\nc{\Coind}{\operatorname{Coind}}
\nc{\wt}{\operatorname{wt}}
\nc{\hwt}{\operatorname{wt}}
\nc{\rk}{\operatorname{rank}}
\nc{\ch}{\operatorname{ch}}
\nc{\sch}{\operatorname{sch}}
\nc{\mdim}{\operatorname{mdim}}
\nc{\Stab}{\operatorname{Stab}}
\nc{\Ima}{\operatorname{Im}}
\nc{\Irr}{\operatorname{Irr}}
\nc{\Spec}{\operatorname{Spec}}
\nc{\Res}{\operatorname{Res}}
\nc{\res}{\operatorname{res}}
\nc{\Aut}{\operatorname{Aut}}
\nc{\Ext}{\operatorname{Ext}}
\nc{\Prec}{\operatorname{Prec}}
\nc{\Fract}{\operatorname{Fract}}
\nc{\gr}{\operatorname{gr}}
\nc{\diag}{\operatorname{diag}}
\nc{\deff}{\operatorname{def}}
\nc{\Vect}{\operatorname{Vect}}
\nc{\HC}{\operatorname{HC}}
\nc{\dpth}{\operatorname{dpth}}
\nc{\sw}{\operatorname{sw}}
\nc{\red}{\operatorname{red}}
\nc{\pari}{\operatorname{par}}
\nc{\pos}{\operatorname{pos}}

\nc{\Cl}{\mathcal{C}\ell}

\nc{\wdchi}{\widetilde{\chi}}
\nc{\wdH}{\widetilde{H}}
\nc{\wdN}{\widetilde{N}}
\nc{\wdM}{\widetilde{M}}
\nc{\wdO}{\widetilde{O}}
\nc{\wdR}{\widetilde{R}}

\nc{\wdV}{\widetilde{V}}

\nc{\wdC}{\widetilde{C}}

\nc{\zero}{\operatorname{zero}}
\nc{\nonzero}{\operatorname{nonzero}}

\nc{\Obj}{\operatorname{Obj}}
\nc{\Dglie}{\operatorname{{\mathcal D}glie}}
\nc{\Fin}{\operatorname{{\mathcal F}in}}
\nc{\pr}{\operatorname{pr}}
\nc{\Adm}{\operatorname{\mathcal{A}dm}}

\nc{\Sg}{{\cS(\fg)}}
\nc{\Shg}{{\cS(\fhg)}}
\nc{\Ug}{{\cU(\fg)}}
\nc{\Uhg}{{\cU(\fhg)}}
\nc{\Sh}{{\cS(\fh)}}
\nc{\Uh}{{\cU(\fh)}}
\nc{\Uhh}{{\cU(\fhh)}}
\nc{\Zg}{{{\mathcal{Z}}(\fg)}}

\nc{\Vir}{{\mathcal{V}ir}}
\nc{\NS}{{\mathcal{N}S}}

\nc{\tZg}{{\widetilde{\mathcal Z}({\mathfrak g})}}
\nc{\Zk}{{\mathcal Z}({\mathfrak k})}

\newcommand{\CC}{\mathbb{C}}

\nc{\Up}{{\mathcal U}({\mathfrak p})}
\nc{\Ah}{{\mathcal A}({\mathfrak h})}
\nc{\Ag}{{\mathcal A}({\mathfrak g})}
\nc{\Ap}{{\mathcal A}({\mathfrak p})}
\nc{\Zp}{{\mathcal Z}({\mathfrak p})}
\nc{\cR}{\mathcal R}
\nc{\cS}{\mathcal S}
\nc{\cP}{\mathcal P}
\nc{\cT}{\mathcal{T}}
\nc{\cA}{\mathcal A}
\nc{\cU}{\mathcal U}
\nc{\cZ}{\mathcal Z}
\nc{\cM}{\mathcal M}
\nc{\cL}{\mathcal L}
\nc{\cF}{\mathcal F}
\nc{\fg}{\mathfrak g}
\nc{\cB}{\mathcal{B}}

\nc{\fo}{\mathfrak o}

\nc{\CO}{\mathcal O}
\nc{\CR}{\mathcal R}

\nc{\cK}{\mathcal{K}}
\nc{\cW}{\mathcal{W}}
\nc{\bM}{\mathbf{M}}
\nc{\bL}{\mathbf{L}}
\nc{\bN}{\mathbf{N}}

\nc{\zq}{\mathpzc q}

\nc{\fl}{\mathfrak l}
\nc{\fn}{\mathfrak n}
\nc{\fm}{\mathfrak m}
\nc{\fp}{\mathfrak p}
\nc{\fh}{\mathfrak h}
\nc{\ft}{\mathfrak t}
\nc{\fk}{\mathfrak k}
\nc{\fb}{\mathfrak b}
\nc{\fs}{\mathfrak s}
\nc{\fB}{\mathfrak B}

\nc{\vareps}{\varepsilon}
\nc{\varesp}{\varepsilon}
\nc{\veps}{\varepsilon}

\nc{\fsl}{\mathfrak{sl}}
\nc{\fgl}{\mathfrak{gl}}
\nc{\fso}{\mathfrak{so}}
\nc{\fosp}{\mathfrak{osp}}
\nc{\fsp}{\mathfrak{sp}}
\nc{\fq}{\mathfrak q}
\nc{\fsq}{\mathfrak{sq}}
\nc{\fpsq}{\mathfrak{psq}}
\nc{\fpq}{\mathfrak{pq}}

\nc{\fhg}{\hat{\fg}}
\nc{\fhn}{\hat{\fn}}
\nc{\fhh}{\hat{\fh}}
\nc{\fhb}{\hat{\fb}}
\nc{\hrho}{\hat{\rho}}

\nc{\hsl}{\hat{\fsl}}
\nc{\fpo}{\mathfrak{po}}
\nc{\dirlim}{\underset{\rightarrow}{\lim}\,}
\nc{\nen}{\newenvironment}
\nc{\ol}{\overline}
\nc{\ul}{\underline}
\nc{\ra}{\rightarrow}
\nc{\lra}{\longrightarrow}
\nc{\Lra}{\Longrightarrow}
\nc{\bo}{\bar{1}}
\nc{\Lla}{\Longleftarrow}

\nc{\Llra}{\Longleftrightarrow}

\nc{\thla}{\twoheadleftarrow}

\nc{\lang}{(}
\nc{\rang}{)}

\nc{\hra}{\hookrightarrow}

\nc{\iso}{\overset{\sim}{\lra}}

\nc{\ssubset}{\underset{\not=}{\subset}}

\nc{\vac}{|0\rangle}

\nc{\simka}{{\ \scriptscriptstyle _{{\sim}}^\text{\tiny{k}}\ }}

\nc{\Thm}[1]{Theorem~\ref{#1}}
\nc{\Prop}[1]{Proposition~\ref{#1}}
\nc{\Lem}[1]{Lemma~\ref{#1}}
\nc{\Cor}[1]{Corollary~\ref{#1}}
\nc{\Conj}[1]{Conjecture~\ref{#1}}
\nc{\Claim}[1]{Claim~\ref{#1}}
\nc{\Defn}[1]{Definition~\ref{#1}}
\nc{\Exa}[1]{Example~\ref{#1}}
\nc{\Rem}[1]{Remark~\ref{#1}}
\nc{\Note}[1]{Note~\ref{#1}}
\nc{\Quest}[1]{Question~\ref{#1}}
\nc{\Hyp}[1]{Hypoth\`ese~\ref{#1}}
\nen{thm}[1]{\label{#1}{\bf Theorem.\ } \em}{}
\nen{prop}[1]{\label{#1}{\bf Proposition.\ } \em}{}
\nen{lem}[1]{\label{#1}{\bf Lemma.\ } \em}{}
\nen{cor}[1]{\label{#1}{\bf Corollary.\ } \em}{}
\nen{conj}[1]{\label{#1}{\bf Conjecture.\ } \em}{}

\nen{claim}[1]{\label{#1}{\bf Claim.\ } \em}{}

\nen{defn}[1]{\label{#1}{\bf Definition.\ } }{}
\nen{exa}[1]{\label{#1}{\bf Example.\ } }{}
\nen{rem}[1]{\label{#1}{\em Remark.\ } }{}
\nen{exer}[1]{\label{#1}{\em Exercise.\ } }{}
\nen{sket}[1]{\label{#1}{\em Sketch of proof.\ } }{}
\nen{quest}[1]{\label{#1}{\bf Question.\ } \em}{}

\nen{hyp}[1]{\label{#1}{\bf Hypoth\`ese.\ } \em}{}
\begin{document}
\setcounter{section}{0}
\setcounter{tocdepth}{1}
\numberwithin{equation}{section}

\begin{center}{\Large \textbf{
Construction of matryoshka nested indecomposable $N$-replications of Kac-modules of quasi-reductive Lie superalgebras, including the 
$sl(m/n)$,  $osp(2/2n)$ series.\\
  }}\end{center}

\begin{center}
  Jean Thierry-Mieg\textsuperscript{1$\star$},
  Peter Jarvis\textsuperscript{2,3},
  and 
  Jerome Germoni\textsuperscript{4},\\
  with an appendix by Maria Gorelik\textsuperscript{5}

\end{center}

\begin{center}
{\bf 1} NCBI, National Library of Medicine, National Institute of Health, \\
  8600 Rockville Pike, Bethesda MD20894, U.S.A.
\\
{\bf 2} 
School of Natural Sciences (Mathematics and Physics),\\
  University of Tasmania, Private Bag 37,
  Hobart, Tasmania 7001, Australia.\\
{\bf 3} Alexander von Humboldt Fellow.
\\
  {\bf 4} Université Claude Bernard Lyon 1, CNRS UMR 5208,\\
  Institut Camille Jordan, F-69622 Villeurbanne, France.
  \\
    {\bf 5} Department of Mathematics, the Weizmann Institute of Science, Rehovot, Israel.
    \\
    
* mieg@ncbi.nlm.nih.gov, peter.jarvis@utas.edu.au,germoni@math.univ-lyon1.fr,maria.gorelik@weizmann.ac.il
\end{center}

\begin{center}
\today
\end{center}


\definecolor{palegray}{gray}{0.95}
\begin{center}
\colorbox{palegray}{
  \begin{tabular}{rr}
  \begin{minipage}{0.1\textwidth}
  \end{minipage}
  &
  \begin{minipage}{0.85\textwidth}
    \begin{center}
    {\it 34th International Colloquium on Group Theoretical Methods in Physics}\\
    {\it Strasbourg, 18-22 July 2022} \\
    \doi{10.21468/SciPostPhysProc.?}\\
    \end{center}
  \end{minipage}
\end{tabular}
}
\end{center}

\section*{Abstract}
{\bf
  We construct a new class of finite dimensional indecomposable representations of simple superalgebras
  which may explain, in a natural way, the existence of the heavier elementary particles.
  In type I Lie superalgebras sl(m/n) and osp(2/2n) , one of the Dynkin weights labeling
  the finite dimensional irreducible representations is continuous. Taking the derivative, we show how
  to construct 
  indecomposable representations recursively embedding N copies of the
  original irreducible representation, coupled by generalized Cabibbo angles, as observed among
  the three generations of leptons and quarks of the standard model. The construction is then
  generalized in the appendix to quasi-reductive Lie superalgebras.
}


\section{Introduction}
\label{sec:intro}

In Kac's complete classification of the simple Lie superalgebras \cite{kac1975classification,Kac1977characters},
two families contain an even generator $y$ commuting with the even subalgebra, namely the 
$A(m-1,n-1) = sl(m/n),\;m\neq n$ and the 
$C(n+1) = osp(2/2n)$ superalgebras. They admit a single Dynkin diagram with a single odd positive simple root $\beta$ \cite{FSS89} .

The even subalgebra, in the corresponding Chevalley basis, has the structure:
\BE
\label{eqn:1.1}
   [h_i,h_j] = 0\;,\;\;[h_i,e_j] = C_{ij} e_j\;,\;\;[h_i,f_j] = -C_{ij} f_j\;,
\\ 
    \; [ y, h_i] = [y,e_i] = [y,f_i] = 0\;,\;\;\;\;i,j = 1,2,...r
\EE
where $r$, $h_i$, $e_i$, $f_i$ and $C_{ij}$ denote respectively the rank, 
the Cartan commuting generators, 
the raising and the lowering generators associated to the simple roots,
and the Cartan matrix of
the semisimple even Lie subalgebra $sl(m) \oplus sl(n)$, respectively $sp(2n)$,
with rank $r = m + n - 2$, respectively $r = n$. 
The remaining raising (respectively lowering) generators of the even semisimple subalgebra are generated by the
iterated commutators of the $e$ (respectively $f$) generators limited by the 
Serre rule $ad(e_i)(e_j)^{-C_{ij}+1}=0$. 
Finally, the additional even generator $y$,
that physicists often call the hypercharge, centralizes the even subalgebra. 
Even in finite dimensional representations, $y$ is not quantized, and as shown below, this
is the cornerstone of our new construction of nested indecomposable $N$-replications of an arbitrary Kac module
which we propose to call matryoshka representations.

In its odd sector, the superalgebra has $P$ odd raising generators $u_i$ corresponding to
the $P$ positive odd roots $\beta_i$ and $P$ odd lowering generators $v_i$ corresponding to the $-\beta_i$, with
$P=mn$ for $sl(m/n)$, or $P=2n$ for $osp(2/2n)$. In both cases, the $u_i$ 
sit in the irreducible fundamental representation of the even subalgebra. We call $u_1$ the lowest
weight vector of the $u_i$ representation; $u_1$ corresponds to the simple positive odd root $\beta = \beta_1$. 
Reciprocally, we call $v_1$ the highest weight vector of the $v_i$.
For our following analysis, the important relations are
\BE
\label{eqn:1.2}
[y,u_i] = u_i\;,\;\;[y,v_i] = - v_i\;,\;
\\
\{u_i,u_j\}=\{v_i,v_j\}=0,  
\\
\{u_i,v_j\} = d^a_{ij}\,\mu_a + ky \,\delta_{ij}\;,\;\;
\EE
where $d^a_{ij}$ and $k$ are constants $(k \neq 0)$ and the $\mu_a$ span the even generators of type $(h,e,f)$. 
That is: the hypercharge $y$ grades the superalgebra, with eigenvalues $(0,\pm 1)$.
The $u_i$ anticommute with each other. So do the $v_i$.
Finally and most important, the anticommutator of the odd raising 
operator $u_i$ with the odd lowering 
operator $v_i$ corresponding to the opposite odd root
depends linearly on the hypercharge $y$.
In particular, $\{u_1,v_1\} = h_{\beta} = d^a_{11} h_a + ky$, where  $k$ is non zero and
$h_\beta$ is the Cartan generator associated to the odd simple root $\beta$.
See for example the works of Kac \cite{kac1975classification,Kac1977characters} or
the dictionary on superalgebras by Frappat, Sciarrino and Sorba \cite{FSS96} for
details.

\section {Construction of the Kac modules}
\label{sec:kacmodule}

Following Kac \cite{Kac1977characters},
choose a highest weight vector $\LX$ defined as an eigenstate of the Cartan generators $(h_i, y)$,
and annihilated by all the raising generators $(e_i,u_j)$. 
The eigenvalues $a_i$  of the Cartan operators $h_i$ are called the even Dynkin labels.
The eigenvalue $b$ of the Cartan operator $h_{\beta}$ corresponding to the odd simple root is called
the odd Dynkin label:
\BE
\label{eqn:2.1}
h_i \LX = a_i \LX\;,\;\;\{u_1,v_1\} \;\LX = h_{\beta}\LX = b\LX.
\EE

Construct the corresponding Verma module using the free action on
$\LX$ of the lowering generators $(f,v)$ 
modulo the commutation relations of the superalgebra. 
Since the $v$ anticommute, the polynomials in $(f,v)$ acting on $\LX$ are at 
most of degree $P$ in $v$, and hence the Verma module is graded 
by the hypercharge $y$ and contains exactly $P$ layers.

Consider the antisymmetrized product $w^-$ of all the odd lowering generators $(v_i, i=1,2,...,P)$.
The state $\LXB = w^- \LX$ is a highest
weight with respect to the even subalgebra $e_i \LXB = 0$. 
Indeed $e_i$ annihilates $\LX$ and each
term in the Leibniz development of $[e_i,w^-]$ contains a repetition of one
of the $v$ generators, and hence vanishes. 

Let $\rho$ be the half supersum of the even and odd positive roots
\BE
\label{eqn:2.2}
\rho = \rho_0 - \rho_1 = \frac{1}{2} (\sum \alpha^+  - \sum \beta^+)\;.
\EE
Let $w^+$ be the antisymmetrized product of all the $u$ generators. 
As shown by Kac \cite{Kac1977characters},
we have
\BE
\label{eqn:2.3}
w^+\LXB = w^+w^- \LX = \pm \prod_i <\LX + \rho | \beta_i> \LX
\EE
where the product iterates over the $P$ positive odd roots $\beta_i$, the sign
depends on the relative ordering of $w^+$ and $w^-$ and the bilinear form $<|>$
is a symmetrized version of the Cartan metric. 
If this product is non-zero, the Verma module is called typical. 
$\LX$ belongs to the orbit of the $\LXB$ and {\it vice-versa},  hence
they both belong to the same irreducible submodule.
If the scalar product $<\LX + \rho | \beta_i>$ vanishes for one or more odd positive 
root $\beta_i$, the Verma module is no longer irreducible but only
indecomposable since $\LX$ is not in the orbit of $\LXB$.
It is then called atypical of type $i$ and there exists a state
$\omega_i$ with Cartan eigenvalues $\LX_i - \beta_i$ which is a sub highest weight
annihilated by all the even and odd raising operators $(e,u)$.
In the present study, we do not quotient out by this submodule but
preserve the indecomposable Verma module construction because we want to
preserve the continuity in $b$.
Notice that in the $A$ and $C$ superalgebras that we are studying the odd roots are
on the light-cone
of the Cartan root space : $< \beta_i | \beta_i> = 0$.
Therefore, if $\LX$ is atypical $i$, the secondary highest weight 
$\LX - \beta_i$ is also atypical $i$.

As in the Lie algebra case, this Verma module is infinite dimensional,
because of the acceptable iterated action of the even lowering generators $f$.
But as we just discussed, the iterated action of the anticommuting odd lowering operators $v$
saturates at layer $P$.

Let us now recall for completeness the usual procedure to
extract a finite dimensional irreducible module from a Lie algebra Verma module. 
All the states with negative even Dynkin labels 
which are annihilated by the even raising generators
can be quotiented.
For example, given a Chevalley basis $(h,e,f)$ for the Lie algebra $sl(2)$
and a Verma module with highest weight $\LX$, we have
\BE
\label{eqn:2.4}
[h,e] = e\;,\;\;[h,f] = -f\;,\;\;[e,f] = 2 h\;,
\\
 h \LX = a \LX\;,\;\; e\LX = 0\;,
 \EE
hence 
\BE
\label{eqn:2.5}
  h f^n \LX = (a - 2n) f^n \LX\;,\;\;  ef^n \LX =   n(a -n +1) f^{n-1} \LX\;.
\EE
If $a$ is a positive integer, the Verma module can be
quotiented by the orbit of the state $f^{a+1}\LX$,
and  the equivalence classes form 
an irreducible module of finite dimension $a+1$.

Generalizing to a superalgebra, all the even Dynkin labels $a_i$
associated to the Cartan operators $h_i, i = 1,2,...,r$ are
restricted to non negative 
integers. We pass to the  quotient in each even submodule and
define the Kac module as the resulting 
finite dimensional quotient space. The crucial observation is that the identification of the
even sub highest weights $\omega$ requests to solve a set of equations involving the
even Dynkin labels $a_i$, but independent of the odd Dynkin weight $b$, which
remains non-quantized. For example, in $sl(2/1)$, the state $\omega = (a\,fv - (a+1)vf)\LX$
is an even highest weight \cite{TMJJ2022a}.
But please remember that we do not quotient out the atypical submodules.

Note that this procedure does not extend to the type II Lie-Kac superalgebras $B(m,n)$, $D(m,n)$, $F(4)$ and $G(3)$,
because these algebras contain even generators with hypercharge $y= \pm 2$.
For example the generator associated to the lowest weight of the adjoint representation.
Indeed, the  supplementary root of 
the (affine) extended Dynkin diagram is even. Thus, the Kac module is finite dimensional
if and only if its hidden extended Dynkin label is also a non negative
integer. This integrality constraint
involves $b$. So the representations of the type II superalgebras
are finite dimensional only for quantized values of $b$, see
Kac \cite{Kac1977characters} for the original proofs and \cite{TM83,TM85} for examples. 

The remaining even highest weights $\LX^p_{(...)}$ are spread over the 
$P$ layers with hypercharge decreasing from $y$ down to $y - P$. On the zeroth layer,
we have $\LX^0 = \LX$, 
on the first layer we have the $P$ weights $\LX^1_i = \LX - \beta_i$,
on the second layer 
we have the $P(P-1)/2$ weights $\LX^2_{ij} = \LX - \beta_i - \beta_j$, $i \neq j$,
down to the $P^{th}$ layer $\LX^P_{12...P} = \LXB$,
each time excluding the even highest weight vectors with negative even Dynkin labels, 
since they have been quotiented out.
For an explicit construction of the matrices of the indecomposable representations of $sl(2/1)$, 
we refer the reader to our study \cite{TMJJ2022a} and references therein.

To conclude, if the Kac module with highest weight $\LX$ is typical, it is irreducible. 
If it is atypical, it is indecomposable.
In both cases, its even highest weights are the $\LX^p_{(...)}$ and the whole module is given by the
even orbits of the  $\LX^p_{(...)}$ with non negative even Dynkin labels and hypercharge $y-p$.

\section{On the derivative of the odd raising generators}
\label{sec:derivative}

Consider a finite $D$ dimensional Kac module with highest weight $\LX$, typical or atypical,
as described in the previous section. Call $a_i$ the even Dynkin labels and $b$ the odd Dynkin label
and $y$ the eigenvalue of the hypercharge $Y$ acting on the highest weight state.
Notice that $y$ is a linear combination of $b$ and the even Dynkin weights $a_i$.
As shown above, in our Chevalley basis,
the matrices representing the $(e,f,v)$ generators in the Verma module are by construction independent of $b$,
and the matrix representing the hypercharge generator $Y$ can be written as
$Y = y I + \alpha^i h_i$, where $I$ is the identity, $h_i$ the even Cartan generators of the
semi simple even subalgebra (i.e. excluding $Y$) and the $\alpha^i $ are constants independent of $y$.
This remains true in the Kac module because the quotient operations needed to pass to
the finite dimensional submodule does not involve $b$.
Finally, the matrices representing the odd raising generators $u$ are linear in $b$, i.e. in $y$,
because, when we push an odd raising generator $u$ acting from the left through an element of the Kac module,
i.e. through a polynomial in $(f,v)$ acting on $\LX$, we must contract $u$ with one of the $v$ generators
before $u$ touches $\LX$.

Now consider the derivatives $u'_i$ of the odd raising $u_i$ matrices
\BE
\label{eqn:3.1}
u'_i(a) = \partial_y u_i(a,y)\;.
\EE
Using \eqref{eqn:1.2}, we derive the anticommutation relations
\BE
\label{eqn:3.2}
\{u'_i, v_j\} = \partial_y \{u_i,v_j\} = \partial_y (d^a_{ij} \mu_a + k y \delta_{ij}) = k \delta_{ij}\;,
\EE
where the $\mu$ matrices span the even generators $(h,e,f)$, and where $\mu(a)$ and $v(a)$ are independent of $y$.
Another way of seeing the same results is to compute the $\{u_i(a,y),v_j(a)\}$  anticommutator,
divide by $y$ and take the limit when $y$ goes to
infinity.
Since the matrix elements of the even generators are all bounded when
$y$ diverges,
except the hypercharge $Y$ with spectrum $y, y-1, ..., y-P$,
we arrive at the same conclusion: the $\{u',v\}$ anticommutator is proportional
to the identity on the whole Kac module. Many explicit examples of the matrices $u(a,y), u'(a),v(a)$
can be found in our extensive study of $sl(2/1)$ \cite{TMJJ2022a}.

This result holds for the Verma modules, for the typical-irreducible Kac modules
and for the atypical-indecomposable Kac modules of type I superalgebras,
but does not hold for the type II superalgebras or for the 
irreducible atypical modules of the type I superalgebras because
we need continuity in $b$. Indeed, we proved in \cite {JTM2022a} by
a cohomology argument 
that the fundamental atypical triplet of $sl(2/1)$ cannot be doubled.

The procedure does not hold for the simple superalgebras $psl(n/n)$.
It works for $sl(n/n)$, but this superalgebra is not simple because if $m = n$ the $(m/n)$ identity
operator $Y$ is supertraceless and generates an invariant 1-dimensional subalgebra that can be quotiented out.
The resulting simple superalgebra $psl(n/n)$ corresponds to the quantized case $y=0$ and we cannot take the derivative.

\section{Construction of an indecomposable N-replication of a Kac module}
\label{sec:matrio}

Given a finite $D$ dimensional Kac module, typical or atypical-indecomposable, represented by $D \times D$ matrices
$(\mu,y,u,v)$, constructed as above and where $\mu$ collectively denotes the even matrices of type $(h,e,f)$,
consider the doubled matrices of dimension $2D \times 2D$ :
\BE
\label{eqn:4.1}
M = \begin{pmatrix} 
   \mu & 0 \cr
   0 & \mu
\end{pmatrix}
\;,\;\;
Y = \begin{pmatrix} 
   y & I \cr
   0 & y
\end{pmatrix}
\;,\;\;
U = \begin{pmatrix} 
   u & u' \cr
   0 & u
\end{pmatrix}
\;,\;\;
V = \begin{pmatrix} 
   v & 0 \cr
   0 & v
\end{pmatrix}
\;,
\EE
where we used the $D \times D$ matrices $u'$ constructed in the previous section.
By inspection, the matrices $(M,Y,U,V)$ have the same super-commutation relations as the matrices $(\mu,y,u,v)$
and therefore form an indecomposable representation
of the same superalgebra of doubled dimension $2D$. This representation cannot be diagonalized since the
matrix $Y$ representing the hypercharge cannot be diagonalized because of its block Jordan structure.

The block $u'$ can be rescaled via a change of variables
\BE
\label{eqn:4.2}
Q = \begin{pmatrix} 
   \lX & 0\cr
   0 & 1
\end{pmatrix}
\;,\;\;
Q^{-1} = \begin{pmatrix} 
   1/\lX & 0\cr
   0 & 1
\end{pmatrix}
\;,
\\
\\
QUQ^{-1} = \begin{pmatrix} 
   u & \lX u' \cr
   0 & u
\end{pmatrix}
\;,\;\;
QYQ^{-1} = \begin{pmatrix} 
   y & \lX I \cr
   0 & y
\end{pmatrix}
\;.
\EE

Furthermore, we can construct a module of dimension $ND$, for any positive integer $N$ by iterating the
previous construction. By changing variables we can then introduce a complex parameter $\lX$ at each level.
For example, for $N=3$, we can
construct
\BE
\label{eqn:4.3}
M = \begin{pmatrix} 
   \mu & 0 & 0 \cr
   0 & \mu & 0 \cr
   0 & 0 & \mu
\end{pmatrix}
\;,\;\;
Y = \begin{pmatrix} 
   y & I & 0 \cr
   0 & y & I \cr
   0 & 0 & y
\end{pmatrix}
\;,\;\;
\tilde{Q}Y\tilde{Q}^{-1} = \begin{pmatrix} 
   y & \lX_1\,I & 0 \cr
   0 & y & \lX_2\,I \cr
   0 & 0 & y
\end{pmatrix}
\;,
\\
\\
U = \begin{pmatrix} 
   u & u' & 0 \cr
   0 & u & u' \cr
   0 & 0 & u
\end{pmatrix}
\;,\;\;
V = \begin{pmatrix} 
   v & 0 & 0 \cr
   0 & v & 0 \cr
   0 & 0 & v
\end{pmatrix}
\;,\;\;
\tilde{Q}U\tilde{Q}^{-1} = \begin{pmatrix} 
   u & \lX_1 u' & 0 \cr
   0 & u & \lX_2 u' \cr
   0 & 0 & u
\end{pmatrix}
\;.
\EE

\textbf{Matryoshka theorem:}
Given any finite dimensional, typical or atypical, Kac module of a type I simple superalgebra, $A(m/n),\;m\neq n$ or $C(n)$,
using the derivative $u'$ of the odd raising generators with respect to the hypercharge $y$
which centralizes the even subalgebra,
we can construct an indecomposable representation recursively embedding $N$ replications of the original
module.

We propose the name matryoshka because this nested structure strongly resemble the famous Russian dolls.

\section{Conclusion}
\label{sec:conclusion}

Representation theory of Lie algebras and superalgebras involves three increasingly difficult steps:
classification, characters and construction. In Lie algebra theory, we can rely on three major results:
all finite dimensional representations of the semisimple Lie algebras are completely reducible,
their irreducible components are classified by the Dynkin labels of their highest weight state, 
their characters are given by the Weyl formula. Nevertheless, the actual construction
of the matrices, although known in principle, remains challenging.
We only know the matrices in closed analytic form in the case of $sl(2)$.

Finite dimensional simple Lie superalgebras have been classified by Kac \cite{kac1975classification}.
In the present study, we only consider the simple basic classical superalgebras of type 1,
$sl(m/n),\;m\neq n$ and $osp(2,2n)$ which are characterized by
the existence of an even generator, the hypercharge $y$, commuting with the even subalgebra.
As for Lie algebras, their irreducible modules can be
classified by the Dynkin labels of their highest weight and Kac \cite{Kac1977characters} has discovered
in 1977 an elegant generalization of the Weyl formula.

But there are two additional difficulties. First, as found by Kac, the hypercharge $y$
of the finite dimensional modules is not quantized, but for certain discrete values, the Kac module
ceases to be irreducible but becomes indecomposable. One can quotient out one or several
invariant submodules and the Weyl-Kac formula of the irreducible quotient module
is not known in general \cite{JKTM90,KacWakimoto94}. Furthermore, there is a rich zoology of finite dimensional indecomposable modules
which were progressively discovered by Kac \cite{Kac1977characters}, Scheunert \cite{ScheunertNahmRittenberg1977irreps}, Marcu \cite{marcu1980tensor,marcu1980representations}, Su\cite{Su1992}, and others, culminating
in the classification of Germoni \cite{germoni1997representations,germoni1998indecomposable}.
See \cite{TMJJ2022a} for an explicit description of the
indecomposable $sl(2/1)$ modules.

A particular class, first described by Marcu \cite{marcu1980representations}, is of great interest in physics because it has
implications for the standard model of leptons and quarks. These particles are well
described by $sl(2/1)$ irreducible modules graded by chirality
 \cite{ne1979irreducible,fairlie1979higgs,dondi1979supersymmetric,NTM1980,TMN1982,Thierry_Mieg_2021}. However,
 experimentally, they appear as a hierarchy of three quasi identical families, for example
the muon and the tau behave as heavy electrons. This hierarchical structure has no clear explanation
in Lie algebra theory. Furthermore, the three families leak into each other in a subtle way
first described by Cabibbo (C) for the strange quarks and generalized to all three families
by Kobayashi and Maskawa (KM). In a certain technical sense, the axis of the electroweak interactions is not
orthogonal to the axes of the strong interactions, but tilted by small angles, called the CKM angles.
As a result the weak interactions are not truly universal because the heavier quarks leak into the lighter quarks.
Again, this experimental phenomenon has no explanation in Lie algebra theory precisely because all representations are completely reducible.

Marcu found in 1980 \cite{marcu1980representations} that the fundamental $sl(2/1)$ quartet can be duplicated and triplicated in an indecomposable way.
Coquereaux, Haussling, Scheck and coworkers \cite{coquereaux1991elementary,ScheckHaussling1998triangular,haussling1998leptonic}
have proposed in the 90's to interpret these representations as a description of the CKM mechanism.
This raises several questions: is the construction of Marcu limited to three generations, as observed experimentally in the case of the quarks and leptons,
or does there exist indecomposable  modules involving more layers? Is this property specific of $sl(2/1)$, or is it applicable to
other simple Lie-Kac superalgebras?

We have previously partially answered these questions.
In \cite{germoni1997representations,germoni1998indecomposable} the existence of multi generations indecomposable modules is indicated.
In \cite{JTM2022a}, we proved, using cohomology, that any Kac module of a type I superalgebra can be duplicated.
But these were just proofs of existence.

In the present study, using the derivative of the odd generators relative to the hypercharge, we have shown that any Kac module of a
type I Lie-Kac superalgebra $sl(m/n),\;m\neq n$ and $osp(2/2n)$ can be replicated any desired number of times in an indecomposable way.
We have also shown that atypical representations cannot be replicated. We can therefore, in this framework predict the existence of
three species of sterile right neutrinos from the observation of the non-zero PMNS (leptonic CKM) mixing angles.

In the appendix presented below (\ref{sec:appendix}),
we further show that these results are valid for any  Kac module $K(L)$  over a quasi-reductive Lie
superalgebra $\mathfrak{g}$  of type I.
As the reader will notice, the style of this appendix contributed by M.G.
is more general and more abstract.
We hope that this split/joint presentation will appeal
to the wide audience of the G34 conference, equally composed of  mathematicians and physicists.
The main result is that
the "matryoshka N-replication" of the Kac module $K(L)$
has the structure of a module over a Heisenberg superalgebra.

These results are interesting for physics, surprising relative to Lie algebra theory, and very specific
as we actually construct the matrices of these ``matryoshka'' Russian dolls indecomposable modules,
in terms of the matrices of the original Kac module, rather than limit our analysis
to their existence, classification, or the calculation of their characters. 

\section*{Acknowledgments}
The authors heartily thank the organizers of the G34 conference
\paragraph{Author contributions}
JTM, PDJ and JG developed the work on the type 1 superalgebras, MG contributed the generalization to quasi-reductive superalgebras presented in the appendix.
\paragraph{Funding information}
JTM was supported by the Intramural Research Program of the National Library of Medicine, National Institute of Health. MG was supported by ISF 1957/21 grant.

\begin{appendix}

\section{Appendix: Generalization to quasi-reductive Lie superalgebras}
\label{sec:appendix}
\begin{center}
  Contributed by Maria Gorelik.
  \end {center}

A finite-dimensional 
Lie superalgebra  $\fg=\fg_{\ol{0}}\oplus \fg_{\ol{1}}$ is {\em quasi-reductive}  if  $\fg_{\ol{0}}$  is a reductive Lie algebra and $\fg_{\ol{1}}$ is a semisimple $\fg_0$-module. Quasi-reductive Lie superalgebras were introduced in~\cite{serganova2011quasireductive}, \cite{M}. Simple quasi-reductive Lie superalgebras are classical Lie superalgebras in Kac's classification (see~\cite{Kac0}).   For other examples and partial classification of quasi-reductive Lie superalgebras, see ~\cite{serganova2011quasireductive}, \cite{M}, and \cite{IO}.

 As above, the base field is $\CC$ and
 $\fg$ is a quasi-reductive Lie superalgebra with an even Cartan subalgebra $\fh$. 
This means that $\fh$ is a Cartan subalgebra of the reductive Lie algebra
$\fg_{\ol{0}}$ and that $\fg^{\fh}=\fh$.  The only simple quasi-reductive Lie superalgebras
which do not satisfy this assumption are $Q$-type superalgebras.
We  denote by  $\fh'$  the center of the reductive Lie algebra $\fg_{\ol{0}}$;  one has
$$\fh=\fh'\times \fh''\ \text{ where } \fh'':=[\fg_{\ol{0}},\fg_{\ol{0}}]\cap \fh.$$ 
We
identify $(\fh')^*$ with  the subspace $(\fh'')^{\perp}=\{\nu\in\fh^*|\ \nu(h)=0\ \ \text{ for  any } h\in \fh''\}$.
One has
 $$\fg_{\ol{0}}=[\fg_{\ol{0}},\fg_{\ol{0}}]\times \fh'.$$
 
For an $\fh$-module $N$ we denote by $N_{\nu}$ the generalized weight space corresponding to
$\nu\in\fh^*$:
$$N_{\nu}=\{v\in N|\ \forall h\in\fh\ \ (h-\nu(h))^s v=0 \text{ for } s>>0\}.$$
All modules in this section are assumed to be
 {\em locally finite over $\fh$ with generalized finite-dimensional weight spaces}: this means that 
$N=\oplus_{\nu} N_{\nu}$ and $\dim N_{\nu}<\infty$
for all $\nu\in\fh^*$.  We set
$$\ch N:=\sum_{\nu} \dim N_{\nu} e^{\nu}.$$

A  quasi-reductive Lie superalgebra is of  {\em type I}
if $\fg=\fg_{-1}\oplus \fg_0\oplus \fg_1$ is a $\mathbb{Z}$-graded superalgebra. In this case
 $\fg_{\ol{0}}=\fg_0$  is a reductive Lie algebra,
 $\fg_{\pm 1}$ are odd commutative subalgebras of $\fg$
and $\fh$ acts diagonally on $\fg_{\pm 1}$. Examples of quasi-reductive Lie superalgebra of  type I include $\fgl(m|n)$, $\mathfrak{osp}(2|2n)$, 
$\mathfrak{p}_n$
and others.

\subsection{Self-extensions of highest weight modules}
 Let $M$ be a module with the  highest weight $\lambda$
(i.e. $M$ is a quotient of $M(\lambda)$), and let $v_{\lambda}\in M$ be 
the highest weight vector, i.e. the image of the canonical generator
of $M(\lambda)$. Introduce the natural map 
$$\Upsilon_{M}:\Ext^1_{\fg}(M,M)\to\fh^*$$
as follows. Let $0\to M\to^{\phi_1} N\to^{\phi_2} M\to 0$ be 
an exact sequence. Let 
$v:=\phi_1(v_{\lambda})$ and fix $v'\in N_{\lambda}$ such 
that $\phi_2(v')=v_{\lambda}$. Observe that $v,v'$
is a basis of $N_{\lambda}$ and so there exists $\mu\in\fh^*$ such that
 for any $h\in\fh$ one has $h(v')=\lambda(h)v'+\mu(h)v$ 
(i.e., the representation $\fh\to\End(N_{\lambda})$
is  $h\mapsto \begin{pmatrix} \lambda(h) & \mu(h)\\ 0 &\lambda(h)
\end{pmatrix}$). The map $\Upsilon_{M}$ assigns $\mu$ to the exact sequence. 
It is easy to see that $\Upsilon_{M}:\Ext^1(M,M)\to \fh^*$ is injective.

Notice that if $0\to M\to N_1\to M\to 0$ 
and $0\to M\to N_2\to M\to 0$ are two exact sequences then
\begin{equation}\label{cc}
N_1\cong N_2\ \Longleftrightarrow\ 
\Upsilon_{M}(N_1)=c\Upsilon_{M}(N_2)\ \text{ for some }
c\in\mathbb{C}\setminus\{0\}.\end{equation}
If $N$ is an extension of $M$ by $M$ (i.e., $N/M\cong M$)
we denote by $\Upsilon_{M}(N)$ the corresponding one-dimensional subspace
of $\fh^*$, i.e. $\Upsilon_{M}(N)=\mathbb{C}\mu$, where
$\mu$ is the image of the exact sequence

$0\to M\to N\to M\to 0$. 

\subsubsection{}\label{findim}
Let $M$ be a finite-dimensional highest weight module.
Since $\fg_0$ is reductive, the algebra $\fh''$ acts diagonally on
 any finite-dimensional $\fg$-module. Therefore the image of $\Upsilon_M$ 
annihilates $\fh''$, so lies in $(\fh')^*$. In particular, for
 the image of $\Upsilon_M$ is zero and
$\Ext^1(M, M)=0$ if $\fh'=0$. In other words,  the
finite-dimensional highest weight modules
do not admit non-splitting self-extensions if $\fg_{\ol{0}}$ is semisimple
(for instance, if $\fg$ is a basic classical 
Lie superalgebra of type II).

\subsection{Kac modules}
Let $\fg=\fg_{-1}\oplus \fg_0\oplus \fg_1$ be a quasi-reductive superalgebra of type I.

The following useful construction appears in several papers including~\cite{ChenMazorchuk2021}.
For a given $\fg_0$-module $M$, we may extend $M$ trivially to a
$\fg_0+\fg_1$-module and introduce the Kac module 
$$K(M):=\Ind^{\fg}_{\fg_0+\fg_1} M.$$
This defines an exact functor $K: \fg_0-Mod\to \fg-Mod$
 which is  called {\em Kac functor}.   
 It is easy to see that $K(M)$ is indecomposable 
 if and only if $M$ is indecomposable. 
 
 As $\fg_0$-module we have $K(M)\cong M\otimes \Lambda \fg_{-1}$.  Since  $\Lambda \fg_{-1}$
 is a finite-dimensional module with a diagonal action of $\fh$, $M$ is finite-dimensional
 (resp.,  diagonal $\fh$-module) if and only if $K(M)$ is finite-dimensional
 (resp.,  diagonal $\fh$-module) . Moreover,  $M$ is a locally finite $\fh$-module
 with generalized finite-dimensional weight spaces if and only if  $K(M)$ is such a module.
 
 \subsubsection{Self extensions of Kac modules}\label{const}
 Take $\nu\in (\fh')^*$ and let 
 $J_n(\nu)$ be the $(n|0)$-dimensional indecomposable $\fh$-module  
 spanned by $v_1,\ldots,v_n$ with the action  $hv_i=\nu(h)v_{i+1})$, $hv_n=0$.
 ($h$ acts on  $V_2(\nu)$  $\begin{pmatrix} 0& 0\\
   \nu(h) & 0   \end{pmatrix}$). Observe that $hJ_n(\nu)=0$ for all $h\in\fh''$.
 We view $J_n(\nu)$ as $\fg_0$-module with the zero action of $[\fg_0,\fg_0]$.
 
  For any $\fg_0$-module $L$ the product
 $J_n(\nu)\otimes L$ is an indecomposable $\fg_0$-module which admits a filtration of length $n$ with  the factors  isomorphic to $L$. Thus 
  $K(L\otimes J_n(\nu))$ 
 is an indecomposable $\fg$-module  which admits a filtration of length $n$ with 
 the factors  isomorphic to the Kac module $K(L)$; we denote this module by $K(L;n;\nu)$. 
 In particular, $K(L;2;\nu)$ is a self-extension of the Kac module $K(L)$. 
 
 \subsubsection{}
 Let $L$ be a finite-dimensional highest weight $\fg_0$-module. Then $K(L)$ is
 a finite-dimensional highest weight $\fg$-module
 By~\ref{findim},  the image of $\Upsilon_{K(L)}$ lies in  
$(\fh')^*$. Using the above construction  we obtain that the the image of $\Upsilon_M$ 
is equal to $(\fh')^*$.
 
 \subsubsection{}
 \begin{lem}{}
 If $L$ is a $\fg_0$-module, where $\fh$ acts locally finitely with finite-dimensional generalized weight spaces, then  
 $K(L;n;\nu)\cong K(L;n;\mu)$ if and only if 
 $\nu\in \mathbb{C}^*\mu$.
 \end{lem}
 \begin{proof}
 If $\nu\in \mathbb{C}^*\mu$, then
$J_n(\nu)\cong J_n(\mu)$ and thus $K(L;n;\nu)\cong K(L;n;\mu)$.

Conversely,   assume that $K(L;n;\nu)\cong K(L;n;\mu)$.  
As $\fg_0$-modules 
$$K(L;n;\nu)\cong L\otimes J_n(\nu)\otimes\Lambda \fg_{-1}\cong K(L)\otimes J_n(\nu).$$
 Thus for any $\lambda\in\fh^*$ we have
$K(L;n;\nu)_{\lambda}\cong K(L)_{\lambda}\otimes J_n(\nu)$.
Take $\lambda$ such that $K(L)_{\lambda}\not=0$.
Then $K(L;n;\nu)\cong K(L;n;\mu)$ implies
\begin{equation}\label{ashmodules}
K(L)_{\lambda}\otimes J_n(\nu)\cong K(L)_{\lambda}\otimes J_n(\mu)\ \ \text{ 
as $\fh$-modules.}
\end{equation}

We  will use the following fact: if $V_1,V_2$ are two modules
over a one-dimensional Lie algebra $\mathbb{C}x$ with the minimal polynomials
of $x$ on $V_i$ equal to $(x-c_i)^k_i$, then  the minimal polynomial
of $x$ on $V_1\otimes V_2$ equals to $(x-(c_1+c_2))^{k_1+k_2}$. 

Assume that  $\nu\not\in \mathbb{C}^*\mu$. 
Then there exists $h\in\fh$ such that $\nu(h)=0\not=\mu(h)$.
Recall that  $h-\lambda(h)$ acts nilpotently on $K(L)_{\lambda}$, so
the minimal polynomial of $h$ on $K(L)_{\lambda}$ takes the form $(h-\lambda(h))^k$.
By above, the minimal polynomial of $h$ on $K(L)_{\lambda}\otimes J_n(\nu)$ (resp., on
$K(L)_{\lambda}\otimes J_n(\mu)$)
is $(h-\lambda(h))^{k}$ (resp., $(h-\lambda(h))^{k+n}$).  Hence~(\ref{ashmodules})
does not hold: a contradiction.
\end{proof}

 \subsection{Action of the Heisenberg superalgebra}\label{naiveHeisenberg}
Let $\fg$ be a quasi-reductive superalgebra of type I. Retain notation of~\ref{const}.
Let $\iota':\fg_0\to \fh'$ be the projection along the decomposition
$\fg_0=[\fg_0,\fg_0]\times\fh'$. We endow the superspace
$H=\fg_{-1}\oplus \fh'\oplus \fg_1$ by the structure of Lie superalgebra
with the bracket $[-,-]_n$ given by
$$[\fg_1,\fg_1]_n:=[\fg_{-1}, \fg_{-1}]_n:=0,\ \ [a_-,a_+]_n:=\iota'([a_-,a_+]),\ \ 
[h,a_{\pm}]_n:=0$$
for all $a_{\pm}\in\fg_{\pm 1}$ and $h\in\fh'$. Observe that $H$ is a quasi-reductive
superalgebra of type I  (in fact, $H$ is 
the direct product an odd Heisenberg superalgebra and a commutative superalgebra).
For an $\fh'$-module $M$ we denote by $K_H(M)$
the Kac module for $H$ constructed as in~\ref{const}.

Let $L$ be a $\fg_0$-module.
The construction \eqref{eqn:4.1} defines an action of $H$ on a self-extension
of a Kac module $K(L)$; we describe this action below.

\subsubsection{}
Fix $\mu\in(\fh')^*$. 
Consider a one-parameter family of $n$-dimensional $\fg_0$-modules $V_n(t\nu)$
constructed in~\ref{const} (for  $t\in\mathbb{R}$).
Then $K(L;n;t\nu)$ is a  one-parameter family of self-extensions
of $K(L)$: this self-extension is splitting if $t=0$; by~(\ref{cc}), for $t\not=0$ all these modules
are isomorphic. As a vector space 
$K(L;n;t\nu))$ is canonically isomorphic to 
$$V:=L\otimes J_n(\nu)\otimes\Lambda \fg_{-1}.$$
We let $\rho_t: \fg\to \End(V)$ be the representation corresponding to $K(L;n;t\nu)$.

It is easy to see that 
 $\rho_t(u)=\rho_0(u)$  for $u\in [\fg_0,\fg_0]+\fg_{-1}$ and that
for $u\in\fh'+\fg_1$ one has $\rho_t(u)=\rho_0(u)+t \rho'_t(u)$ for some $\rho'_t(u)\in \End(V)$.

\subsubsection{}
\begin{lem}{lemHmod}
Define a linear map $\phi: H\to \fgl(V)$ by 
$$\phi(a_-):= \rho_0(a_-),\ \ \phi(a_+):=  \rho'_t(a_+),\ \ 
\phi(h):=\rho'_t(h)$$ 
for $a_{\pm}\in\fg_{\pm}$ and $h\in\fh'$. Then $\phi$ is 
the  $H$-representation  isomorphic to the Kac module
$K_H(L',n,\nu)$ where $L\cong L'$ as a superspace and the action of
$H_0=\fh'$ on $L'$ is trivial.
\end{lem}
\begin{proof}
Let us check that $\phi$ is a homomorphism of Lie superalgebras.
For $a_{-},b_{-}\in\fg_{-}$ one has $[a_-,b_-]_n=[a_-,b_-]=0$ and
$$[\phi(a_-),\phi(b_-)]=\rho_0([a_-,b_-])=0=\phi([a_-,b_-]_n).$$

By above, $\frac{\partial^2 \rho_t(a)}{\partial^2 t}=0$
for any $a\in\fg$. Therefore for any $a,b\in \fg$ one has
$$0=\frac{\partial^2 \rho_t([a,b])}{\partial^2 t}=[\rho'_t(a),\rho'_t(b)].$$
Taking $a,b\in\fg_1+\fh'$ we get $0=[\phi(a),\phi(b)]=\phi([a,b]_n)$.

One has
$$[\rho_t(a),\rho'_t(b)]+[\rho'_t(a),\rho_t(b)]=\rho'_t ([a,b])$$
Using this formula for $a=a_-$ and $b=a_+$ we obtain
$$[\phi(a_-),\phi(a_+)]=[\rho_t(a_-),\rho'_t(a_+)]=\rho'_t([a_-,a_+])=
\rho'_t (\iota'[a_-,a_+])=\phi ([a_-,a_+]_n).$$ 
Finally, taking $h\in\fh'$ and $a_-\in\fg_-$ we get 
$$[\phi(h),\phi(a_-)]=[\rho'_t(h),\rho_t(a_-)]=\rho'_t ([h,a_-])=0$$
so $[\phi(h),\phi(a_-)]=\phi([h,a_-]_n)$. Hence $\phi$ is 
a homomorphism, so $\phi$ defines a representation of $H$. Denote the corresponding $H$-module
by $N$.
Let us check that the linear isomorphism $L\iso L'$ induces
the $H$-module isomorphism $N\iso  K_H(L',n,\nu)$.
 
 Since $K(L,n,t\nu)$ is a free $\fg_{-1}$-module generated by the subspace $L\otimes V_n(t\nu)$, $N$ is a free $\fg_{-1}$-module generated by the subspace 
 $L'\otimes V_n(t\nu)$. For any $v\in L `\otimes V_n(t\nu)$ one has
  $\rho_t(a_+)(v)=0$, so $\phi(a_+)(v)=0$.  Take $h\in\fh'$.  Let $v_1,\ldots,v_n$ be the standard
   basis of $J_n(\nu)$ (see~\ref{const}).
 For $w\in L$ we have
 $$\rho_t(h)(w\otimes v_i)=\rho_0(h)(w)\otimes v_i+ t\nu(h)  w\otimes v_{i+1}.$$
 so $\phi(h)(w\otimes v_i)=\nu(h) w\otimes v_{i+1}$.
 Hence $N=K_H(L',n,\nu)$ as required.
 \end{proof}

\subsection{Another construction of the action of the Heisenberg superalgebra}
A natural question is to find a more ``natural'' construction  for Heisenberg superalgebra and its action. This can be done as follows.

\subsubsection{}
 Let $\ft$ be any Lie superalgebra. We introduce the increasing filtration 
by $\cF^{0}(\ft)=0$, $\cF^{1}(\ft):=\ft_{\ol{1}}$ and $\cF^{2}(\ft):=\ft$.

The associated graded Lie superalgebra $\tilde{H}:=\gr_{\cF}(\ft)$ is naturally isomorphic to $\ft$ as a vector superspace;
denoting this linear isomorphism by $\iota:\ft\to \tilde{H}$ we obtain the following formulae for
the bracket on $\tilde{H}$:
$$[\tilde{H}_{\ol{0}},\tilde{H}]=0,\ \ [a,b]:=\iota([\iota^{-1}(a),\iota^{-1}(b)]) \text{ if } a,b\in  \tilde{H}_{\ol{1}}.$$
If $\dim \ft<\infty$, then $\tilde{H}$ is quasi-reductive. If $\ft$ is $\mathbb{Z}$-graded and finite-dimensional, then
$\tilde{H}$ is quasi-reductive of type I ($\tilde{H}$ is 
the direct product an odd Heisenberg superalgebra and a commutative superalgebra).

If $N$ is a $\ft$-module generated by a subspace $N'$ we can define a compatible increasing filtration 
on $N$ by setting $\cF^0(N)=N'$ and $\cF^i(N)=\cF^i(\cU(\ft))N'$. The associated graded module
$\gr_{\cF}(N)$ has a structure of $\tilde{H}$-module.

\subsubsection{Application to~\ref{naiveHeisenberg}}
Retain notation of~\ref{naiveHeisenberg}.
Let $\ft:=\fg$ be  quasi-reductive of type I.   Identify  $\tilde{H}_{\ol{0}}$ with $\fg_{\ol{0}}=\fg_0$. Since   $\tilde{H}_{\ol{0}}$
lies in the center of  $\tilde{H}$, any subspace of $\fg_0$ is an ideal of $\tilde{H}$.
It is easy to see that $\tilde{H}/[\fg_0,\fg_0]$ is isomorphic to $H$ constructed in~\ref{naiveHeisenberg}. 

Set $N:=K(L,n,\nu)$.
Fix $h\in\fh'$ such that $\nu(h)=1$. Recall that the $h$ acts on  $J_n(\nu)$ by a Jordan cell: $J_n(\nu)$ is spanned by $v_1, hv_2, h^2v_1,\ldots, h^{n-1} v$.
$L\otimes J_n(\nu)$ is spanned by $L\otimes v_1, h(L\otimes v_1),\ldots, h^{n-1}(L\otimes v_1)$,  so
 $N':=L\otimes v_1$ generates  $N$ over $\fg$.  Define the increasing filtration 
on $N$ as above.  The associated graded module
$\gr_{\cF}(N)$ is a $\tilde{H}$-module.  We have 
$[\fg_0,\fg_0](\cF^i(N))=\cF^i(N)$ for each $i$, so 
$[\fg_0,\fg_0]$ annihilates $\gr_{\cF}(N)$. Hence
$\gr_{\cF}(N)$ is an $H$-module. It is not hard to see that 
this module is isomorphic to the $H$-module constructed in~\Lem{lemHmod}.

Conclusion: the matryoshka N-replication of the Kac module $K(L)$
has the structure of a module over a Heisenberg superalgebra.

\end{appendix}


\bibliography{slmn}

\begin{thebibliography}{10}
\providecommand{\url}[1]{\texttt{#1}}
\providecommand{\urlprefix}{URL }
\expandafter\ifx\csname urlstyle\endcsname\relax
  \providecommand{\doi}[1]{doi:\discretionary{}{}{}#1}\else
  \providecommand{\doi}{doi:\discretionary{}{}{}\begingroup
  \urlstyle{rm}\Url}\fi
\providecommand{\eprint}[2][]{\url{#2}}

\bibitem{kac1975classification}
V.~G. Kac,
\newblock \emph{Classification of simple lie superalgebras},
\newblock Funktsional'nyi Analiz i ego Prilozheniya \textbf{9}(3), 91 (1975),
\newblock \doi{https://doi.org/10.1007/BF01075610}.

\bibitem{Kac1977characters}
V.~Kac,
\newblock \emph{Characters of typical representations of classical lie
  superalgebras},
\newblock Communications in Algebra \textbf{5}(8), 889 (1977),
\newblock \doi{https://doi.org/10.1080/00927877708822201}.

\bibitem{FSS89}
L.~Frappat, A.~Sciarrino and P.~Sorba,
\newblock \emph{Structure of basic lie superalgebras and of their affine
  extensions},
\newblock Communications in Mathematical Physics \textbf{121}, 457 (1989),
\newblock \doi{https://doi.org/10.1007/BF01217734}.

\bibitem{FSS96}
L.~Frappat, Sciarrino and P.~Sorba,
\newblock \emph{Dictionary on lie superalgebras},
\newblock Academic Press \textbf{10} (1996),
\newblock \doi{https://arxiv.org/hep-th/9607161}.

\bibitem{TMJJ2022a}
J.~Thierry-Mieg, P.~D. Jarvis and J.~Germoni,
\newblock \emph{Explicit construction of the finite dimensional indecomposable
  representations of the simple lie-kac $su(2/1)$ superalgebra and their low
  level non diagonal super casimir operators},
\newblock arXiv preprint arXiv:2207.06545  (2022).

\bibitem{TM83}
J.~Thierry-Mieg,
\newblock \emph{Irreducible representations of the basic classical lie
  superalgebras osp(m/2n), su(m/n), ssu(n/n), d(2/1, $\alpha$)), g(3), f(4)},
\newblock In \emph{20th conference on Group Theoretical Methods in Physics},
  pp. 94--98 (1983).

\bibitem{TM85}
J.~Thierry-Mieg,
\newblock \emph{Table of irreducible representations of the basic classical lie
  superalgebras osp(m/2n), su(m/n), ssu(n/n), d(2/1,$\alpha$)), g(3), f(4)},
\newblock Preprints Observatoire de Paris-Meudon and LBL Berkeley, unpublished
  (1983--1985),
\newblock
  \doi{https://ftp.ncbi.nlm.nih.gov/repository/acedb/1985.characters.pdf}.

\bibitem{JTM2022a}
P.~D. Jarvis and J.~Thierry-Mieg,
\newblock \emph{Indecomposable doubling for representations of the type i lie
  superalgebras sl (m/n) and osp (2/2n)},
\newblock Journal of Physics A: Mathematical and Theoretical \textbf{55}(47),
  475206 (2022),
\newblock \doi{https://doi.org/10.1088/1751-8121/aca22d}.

\bibitem{JKTM90}
J.~Van~der Jeugt, J.~Hughes, R.~King and J.~Thierry-Mieg,
\newblock \emph{A character formula for singly atypical modules of the lie
  superalgebra sl (m/n)},
\newblock Communications in Algebra \textbf{18}(10), 3453 (1990),
\newblock \doi{http://dx.doi.org/10.1080/00927879008824086}.

\bibitem{KacWakimoto94}
V.~Kac and M.~Wakimoto,
\newblock \emph{Integrable highest weight modules over affine superalgebras and
  number theory. lie theory and geometry, 415--456},
\newblock Progr. Math \textbf{123}, 415 (1994),
\newblock \doi{https://doi.org/10.1007/978-1-4612-0261-5_15}.

\bibitem{ScheunertNahmRittenberg1977irreps}
M.~Scheunert, W.~Nahm and V.~Rittenberg,
\newblock \emph{Irreducible representations of the osp (2, 1) and spl (2, 1)
  graded lie algebras},
\newblock Journal of Mathematical Physics \textbf{18}(1), 155 (1977),
\newblock \doi{https://doi.org/10.1063/1.523149}.

\bibitem{marcu1980tensor}
M.~Marcu,
\newblock \emph{The tensor product of two irreducible representations of the
  spl (2, 1) superalgebra},
\newblock Journal of Mathematical Physics \textbf{21}(6), 1284 (1980),
\newblock \doi{https://doi.org/10.1063/1.524577}.

\bibitem{marcu1980representations}
M.~Marcu,
\newblock \emph{The representations of spl (2, 1)—an example of
  representations of basic superalgebras},
\newblock Journal of Mathematical Physics \textbf{21}(6), 1277 (1980),
\newblock \doi{https://doi.org/10.1063/1.524576}.

\bibitem{Su1992}
Y.~Su,
\newblock \emph{Classification of finite dimensional modules of the lie
  superalgebra sl (2/1)},
\newblock Communications in Algebra \textbf{20}(11), 3259 (1992),
\newblock \doi{https://doi.org/10.1080/00927879208824514}.

\bibitem{germoni1997representations}
J.~Germoni,
\newblock \emph{Repr{\'e}sentations ind{\'e}composables des superalg{\`e}bres
  de lie sp{\'e}ciales lin{\'e}aires},
\newblock Comptes Rendus de l'Acad{\'e}mie des Sciences-Series I-Mathematics
  \textbf{324}(11), 1221 (1997),
\newblock \doi{https://doi.org/10.1016/S0764-4442(99)80403-0}.

\bibitem{germoni1998indecomposable}
J.~Germoni,
\newblock \emph{Indecomposable representations of special linear lie
  superalgebras},
\newblock Journal of Algebra \textbf{209}(2), 367 (1998),
\newblock \doi{https://doi.org/10.1006/JABR.1998.7520}.

\bibitem{ne1979irreducible}
Y.~Ne'eman,
\newblock \emph{Irreducible gauge theory of a consolidated salam-weinberg
  model},
\newblock Physics Letters B \textbf{81}(2), 190 (1979),
\newblock \doi{https://doi.org/10.1016/0370-2693(79)90521-5}.

\bibitem{fairlie1979higgs}
D.~Fairlie,
\newblock \emph{Higgs fields and the determination of the weinberg angle},
\newblock Physics Letters B \textbf{82}(1), 97 (1979),
\newblock \doi{https://doi.org/10.1016/0370-2693(79)90434-9}.

\bibitem{dondi1979supersymmetric}
P.~Dondi and P.~Jarvis,
\newblock \emph{A supersymmetric weinberg-salam model},
\newblock Physics Letters B \textbf{84}(1), 75 (1979),
\newblock \doi{https://doi.org/10.1016/0370-2693(79)90652-X}.

\bibitem{NTM1980}
Y.~Ne'eman and J.~Thierry-Mieg,
\newblock \emph{Geometrical gauge theory of ghost and goldstone fields and of
  ghost symmetries},
\newblock Proceedings of the National Academy of Sciences \textbf{77}(2), 720
  (1980),
\newblock \doi{https://doi.org/10.1073/PNAS.77.2.720}.

\bibitem{TMN1982}
J.~Thierry-Mieg and Y.~Ne'eman,
\newblock \emph{Exterior gauging of an internal supersymmetry and su (2/1)
  quantum asthenodynamics},
\newblock Proceedings of the National Academy of Sciences \textbf{79}(22), 7068
  (1982),
\newblock \doi{https://doi.org/10.1073/PNAS.79.22.7068}.

\bibitem{Thierry_Mieg_2021}
J.~Thierry-Mieg,
\newblock \emph{Chirality, a new key for the definition of the connection and
  curvature of a lie-kac superalgebra},
\newblock Journal of High Energy Physics \textbf{2021}(1) (2021),
\newblock \doi{10.1007/jhep01(2021)111}.

\bibitem{coquereaux1991elementary}
R.~Coquereaux,
\newblock \emph{Elementary fermions and su(2|1) representations},
\newblock Physics Letters B \textbf{261}(4), 449 (1991),
\newblock \doi{https://doi.org/10.1016/0370-2693(91)90455-Y}.

\bibitem{ScheckHaussling1998triangular}
R.~H{\"a}ussling and F.~Scheck,
\newblock \emph{Triangular mass matrices of quarks and
  cabibbo-kobayashi-maskawa mixing},
\newblock Physical Review D \textbf{57}(11), 6656 (1998),
\newblock \doi{https://doi.org/10.1103/PhysRevD.57.6656}.

\bibitem{haussling1998leptonic}
R.~Haussling, M.~Paschke and F.~Scheck,
\newblock \emph{Leptonic generation mixing, noncommutative geometry and solar
  neutrino fluxes},
\newblock Physics Letters B \textbf{417}(3-4), 312 (1998),
\newblock \doi{https://doi.org/10.1016/S0370-2693(97)01407-X}.

\bibitem{serganova2011quasireductive}
V.~Serganova,
\newblock \emph{Quasireductive supergroups},
\newblock New developments in Lie theory and its applications \textbf{544}, 141
  (2011),
\newblock \doi{http://dx.doi.org/10.1090/conm/544/10753}.

\bibitem{M}
V.~Mazorchuk,
\newblock \emph{Parabolic category o for classical lie superalgebras},
\newblock In \emph{Advances in Lie superalgebras}, pp. 149--166. Springer,
\newblock \doi{https://doi.org/10.1007/978-3-319-02952-8_9} (2014).

\bibitem{Kac0}
V.~G. Kac,
\newblock \emph{Lie superalgebras},
\newblock Advances in mathematics \textbf{26}(1), 8 (1977),
\newblock \doi{http://dx.doi.org/10.1016/0001-8708(77)90017-2}.

\bibitem{IO}
N.~I. Ivanova and A.~L. Onishchik,
\newblock \emph{Parabolic subalgebras and gradings of reductive lie
  superalgebras},
\newblock Journal of Mathematical Sciences \textbf{152}, 1 (2008),
\newblock \doi{https://doi.org/10.1007/S10958-008-9049-8}.

\bibitem{ChenMazorchuk2021}
C.-W. Chen and V.~Mazorchuk,
\newblock \emph{Simple supermodules over lie superalgebras},
\newblock Transactions of the American Mathematical Society \textbf{374}(2),
  899 (2021),
\newblock \doi{http://dx.doi.org/10.1090/tran/8303}.

\end{thebibliography}
\nolinenumbers

\end{document}